\documentclass[12pt]{amsart}

\usepackage{epsfig}
\usepackage{mathtools}
\usepackage{amsmath}
\usepackage{amssymb}
\usepackage{amsfonts}
\usepackage{MnSymbol}

\usepackage{amscd}
\usepackage{graphicx}
\usepackage{color}
\usepackage{verbatim}
\usepackage{bbm}
\usepackage{upgreek}
\RequirePackage{textcmds}\relax
\ProvideTextCommandDefault{\cprime}{\tprime}
\newtheorem{theorem}{Theorem}[section]
\newtheorem{lemma}[theorem]{Lemma}
\newtheorem{prop}[theorem]{Proposition}

\theoremstyle{definition}
\newtheorem{definition}[theorem]{Definition}

\newtheorem{exam}[theorem]{Example}

\newtheorem{innercustomthm}{Theorem}
\newenvironment{customthm}[1]
  {\renewcommand\theinnercustomthm{#1}\innercustomthm}
  {\endinnercustomthm}

\newtheorem{remark}[theorem]{Remark}

\numberwithin{equation}{section}

\newcommand{\CU}{\mathcal{U}}

\newcommand{\CW}{\mathcal{W}}

\newcommand{\GC}{\mathfrak{C}}
\newcommand{\GP}{\mathfrak{P}}
\newcommand{\GF}{\mathfrak{F}}

\def \N {\mathbb N}

\def \Z {\mathbb Z}
\def \R {\mathcal R}

\def \F {\mathcal F}
\def \H {\mathcal H}

\def \Q {\mathcal Q}

\def \P {\mathcal P}
\def \M {\mathcal M}

\def \htop {h_{\mathsf{top}}}

\def \sq {sequence}

\def \tl {topological}
\def \im {invariant measure}
\def \inv {invariant}

\def \diam {\mathsf{diam}}
\def \htop{h_{\mathsf{top}}}
\def \Hom {\mathsf{Hom}}
\def \id {\mathsf{Id}}

\def \usc {upper semicontinuous}

\def \ens {entropy structure}
\def \ahe {asympto\-ti\-cally $h$-ex\-pan\-sive}
\def \zd {zero-dimen\-sio\-nal}

\def \eps {\varepsilon}

\def \N {\mathbb N}

\def \Z {\mathbb Z}
\def \R {\mathcal R}

\def \F {\mathcal F}
\def \H {\mathcal H}

\def \Q {\mathcal Q}

\def \CU {\mathcal U}
\def \P {\mathcal P}
\def \M {\mathcal M}

\def \htop {h_{\mathsf{top}}}

\def \sq {sequence}

\def \tl {topological}
\def \im {invariant measure}
\def \inv {invariant}

\def \htop{h_{\mathsf{top}}}
\def \Hom {\mathsf{Hom}}
\def \id {\mathsf{Id}}

\def \usc {upper semicontinuous}

\def \ens {entropy structure}

\makeindex

\begin{document}

\title[Tail variational principle and asymptotic $h$-expansiveness]{Tail variational principle and asymptotic $h$-expansiveness for amenable group actions}

\author{Tomasz Downarowicz and Guohua Zhang}

\address{\vskip 2pt \hskip -12pt Tomasz Downarowicz}

\address{\hskip -12pt Faculty of Pure and Applied Mathematics, Wroc\l aw University of Science and Technology, Wybrze\.ze Wyspia\'nskiego 27, 50-370 Wroc\l aw, Poland}

\email{downar@pwr.edu.pl}

\address{\vskip 2pt \hskip -12pt Guohua Zhang}

\address{\hskip -12pt School of Mathematical Sciences and Shanghai Center for Mathematical Sciences, Fudan University, Shanghai 200433, China}

\email{chiaths.zhang@gmail.com}

\parindent=10pt
\begin{abstract}
In this paper we prove the tail variational principle for actions of countable amenable groups. This allows us to extend some characterizations of asymptotic $h$-expan\-si\-ve\-ness from $\mathbb{Z}$-actions to actions of countable amenable groups.
\end{abstract}

\maketitle

\setcounter{tocdepth}{2}

\section{Motivation}

The initial goal of this note was to give a characterization of \ahe\ dynamical systems with actions of countable amenable groups in terms of the existence of a principal symbolic extension, just as it holds for actions of the integers (see \cite{D2005}). It seemed that all needed ingredients are already available. The theory of symbolic extensions for countable amenable group actions as well as the notion of entropy structure have been recently developed in \cite{DownarowiczZhang} (for a brief exposition on the symbolic extension theory in topological dynamics see also a recent survey \cite{DZ-survey} and the references therein). The item linking these two topics, the tail variational principle for such actions (analogous to that known for $\Z$-actions from \cite{D2005}) was given in \cite{Zhou}. But we found that the proof of the tail variational principle in \cite{Zhou} has a gap (the details will be explained in Section \ref{8}). This forced us to deliver our own proof, and due to some unexpected subtleties, this proof grew to become the main part of this paper. We resolve the subtleties by providing a \tl\ analog of the Pinsker formula, known for measure-preserving actions. Other than that, our proof relies on two already existing results on the interplay between \tl\ and measure-theoretic dynamics of countable amenable group actions: the variational principle for topological relative entropy \cite[Theorem 13.3]{DZ} (another proof can be found in \cite[Theorem 5.1]{Yan}) and a characterization of the \tl\ tail entropy in terms of selfjoinings of the action \cite[Theorem 3.1]{ZhouZhangChen}. Eventually, we were able to achieve our initial goal and characterize \ahe\ $G$-actions.

\section{Introduction and statements of the main results}\label{sI}

Throughout this paper, we focus on countable amenable groups $G$, where by ``countable'' we always mean ``infinite countable''. Since the topology on $G$ plays no role, we can as well assume that it is discrete. By $\GF_G$ we denote the collection of all finite nonempty subsets of $G$. Amenability of $G$ is equivalent to the existence of a \emph{F\o lner \sq}, i.e.\  a \sq\ $(F_n)_{n\in \mathbb{N}}\subset\GF_G$ such that $\lim\limits_{n\to\infty}\frac{|F_n\cap gF_n|}{|F_n|}=1$, for any $g\in G$, where  $|\cdot|$ denotes the cardinality~of~a~set. 

Let $X$ be a compact metric space and let $\Hom(X)$ denote the group of all homeomorphisms $\phi:X\to X$. By an \emph{action} (more precisely, \emph{topological action}) of $G$ on $X$ we mean a homomorphism from $G$ into $\Hom(X)$, i.e.\ an assignment $g\mapsto \phi_g$ such that $\phi_{gg'} = \phi_g\circ\phi_{g'}$ for every $g,g'\in G$. It follows automatically that $\phi_e=\id$ (the unit $e$ of $G$ acts by identity) and $\phi_{g^{-1}}=(\phi_g)^{-1}$ for every $g\in G$. Such an action will be denoted by $(X, G)$. Although the group may act on the same space in many different ways, we will usually consider just one such action, hence this notation should not lead to a confusion. An action $(X,G)$ is also called a {$G$-action} or, when the group is understood, a \emph{topological dynamical system} or just a \emph{system}. From now on, to reduce the multitude of symbols used in this paper, we will write $g(x)$ in place of $\phi_g(x)$. The same applies to subsets $A\subset X$: $g(A)$ will replace $\phi_g(A)=\{\phi_g(x):x\in A\}$. A Borel measurable set $ A\subset X$ is called \emph{\inv} if $g(A)=A$ for every $g\in G$. For a Borel probability measure $\mu$ on $X$ and $g\in G$, by $g(\mu)$ we will denote the measure defined by $g(\mu)(A)=\mu(g^{-1}(A))$ for all Borel subsets $A\subset X$. A measure $\mu$ is \emph{invariant} if $g(\mu) = \mu$ for all $g\in G$. The collection of invariant Borel probability measures will be denoted by $\M(X,G)$. By amenability, this collection is nonempty, and for general reasons it is convex and weakly-star compact.

The main result of this paper is Theorem \ref{tail} referred to as the \emph{tail variational principle}, which extends the tail variational principle known for $\mathbb{Z}$-actions (see \cite{D2005}). The meaning of $h^*(X,G)$, $u_1$ and $\theta_k$ is a straightforward adaptation of the analogous terms for $\Z$-actions (see \cite{D2005}) and their definitions will be provided in the next section.
For now, it suffices to say that $h^*(X,G)$ is the \emph{\tl\ tail entropy} introduced (for $\Z$-actions) by M.\ Misiurewicz under the name of \emph{\tl\ conditional entropy} (see \cite{Mi1}), while $u_1$ and $\theta_k$ are special functions on $\M(X,G)$ associated to the so-called \emph{entropy structure} of the action, introduced (again, for $\Z$-actions) by M.\ Boyle and T.\ Downarowicz in \cite{BD} (see also \cite{D2005} for a more detailed exposition on entropy structures).

\begin{theorem}\label{tail}
Let $(X,G)$ be a topological action of a countable amenable group, with finite entropy. Then
$$
h^* (X,G)= \max \{u_1 (\mu): \mu\in \M(X,G)\}= \lim_{k\rightarrow \infty}\ \sup\{\theta_k (\mu):\mu\in\M(X,G)\}.
$$
\end{theorem}

The second equality is elementary. Roughly speaking, it is just swapping the limit of a decreasing \sq\ of functions with the supremum over the domain. In general, such a swapping leads to an inequality, but here we have equality due to upper semicontinuity of the involved functions (see e.g.\ \cite[Proposition 2.4 ]{BD}). The nontrivial part is, of course, the first equality. It is essential that $h^*(X,G)$ is defined exclusively in terms of \tl\ dynamics, while both expressions on the right evidently refer to \im s. So the above variational principle (just like the majority of variantional principles) describes the interplay between \tl\ and measure-theoretic dynamics. For actions of $\Z$ this equality was first proved in \cite{D2005} and an alternative proof was given in \cite{Burguet09}. The proofs do not pass directly to actions of amenable groups, as both of them use some machinery typical for $\Z$-actions (such as Krengel's subadditive ergodic theorem, or Lindenstrauss' theorem on small boundary property of certain systems).

A well-oriented reader will notice that the tail variational principle already exists in the literature, in \cite{Zhou}. But we found that the proof in that paper contains a gap. It seems that Y.\ Zhou attempted to prove Shearer's inequality using plain subadditivity (which is not possible; see Section \ref{8} for a more detailed explanation). There is a subtlety in the notion of conditional entropy: there are two possible versions of how this notion can be defined; in one case the conditioning object is invariant, in the other it is not. It is thus important to distinguish between them, so we will call one of them ``conditional'', and the other ``relative''. The two are connected by the so-called Pinsker formula, which is known in the measure-theoretic case, while the \tl\ version needs to be created from scratch. And this is what we do in the first place. Our proof relies on ``amenable tools'' such as the tilings. 

\medskip
Following the terminology introduced by M.\ Misiurewicz for $\Z$-actions, we will call a system $(X,G)$ \emph{\ahe} if $h^*(X,G)=0$. Because of the (fairly obvious from the definition) inequality $h^*(X,G)\le\htop(X,G)$, any system with zero \tl\ entropy is \ahe. It is also relatively easy to see that any $G$-subshift with a finite alphabet is \ahe\ as well. But the class of \ahe\ systems is much richer. For example, it is known that an algebraic $\mathbb{Z}$-action\footnote{An algebraic action of a discrete group $\Gamma$ is a homomorphism from $\Gamma$ to the group of continuous automorphisms of a compact abelian group $X$.} has finite \tl\ entropy if and only if it is \ahe\ (which is due to Misiurewicz, for details see \cite[Section 7]{Mi1}). This result was extended to algebraic actions $(X,G)$ of countable amenable groups $G$ in \cite[Theorem 7.1]{ChungZhang}. J.~Buzzi provided another prominent class of \ahe\ $\Z$-actions, namely those generated by $C^\infty$ diffeomorphisms on compact smooth manifolds (see \cite{Bu}). If $G$ is a countable amenable group containing $\mathbb{Z}$ as a subgroup of infinite index, then any action $(X,G)$ by $C^1$ maps on a compact smooth manifold (here we allow the manifold to have different dimensions for different connected components, even including zero dimension) has zero topological entropy (see for example \cite[Lemma~5.7]{LT}), and so is trivially \ahe. 

\medskip

For $\Z$-actions there are several conditions equivalent to asymptotic $h$-expansiveness (see e.g.\ \cite{D2005}), one of which is the existence of a \emph{principal symbolic extension}, i.e.\ a \tl\ extension in form of a subshift with a finite alphabet, which preserves the entropy of all \im s. So, a natural question arises: is the same true for actions of countable amenable groups? The notion of a subshift, as well as that of a principal extension is in this context perfectly understandable, so the positive answer would shed a light, in particular, on algebraic actions with finite entropy (such a question was actually raised by Hanfeng Li~\cite{Li}).

As a consequence of Theorem \ref{tail} and the theory of symbolic extensions for amenable group actions developed in \cite{DownarowiczZhang}, we are able to provide a \emph{quasi-positive} answer to this question, with the same restrictions as apply to the entire theory of symbolic extensions for actions of amenable groups (explanations are provided right after the formulation). 

\begin{theorem} \label{main result}
Let $(X,G)$ be a \tl\ action of a countable amenable group. Then the following conditions are equivalent:
 \begin{enumerate}
 \item $(X,G)$ is \ahe.
 \item The entropy structure of $(X,G)$ converges uniformly to the entropy function.
% \item The entropy function is the minimal superenvelope of the entropy % structure.
 \item $(X, G)$ admits a principal quasi-symbolic extension.
 \item For any $\eps>0$ the action admits a quasi-symbolic extension with \tl\ relative 
 entropy at most $\eps$.
 \end{enumerate}
Furthermore, if $G$ is either residually finite or enjoys the comparison property then the quasi-symbolic extensions in the above statements can be replaced by symbolic extensions.
\end{theorem}

We explain, that by a \emph{\tl\ extension} of a system $(X,G)$ we mean another system $(Y,G)$ and continuous surjection $\pi: Y\rightarrow X$ (called the \emph{factor map}) which is equivariant, i.e.\ satisfies the condition $\pi(g(y))= g(\pi(y))$ for each $y\in Y$ and $g\in G$. An extension $(Y,S)$ is \emph{quasi-symbolic} if it is a \tl\ joining of a subshift with a specific zero-dimensional system depending only on the group $G$. This extra system is called a \emph{tiling system} and although it has \tl\ entropy zero (hence is trivially \ahe), it is in general not known whether it admits a symbolic extension at all (let alone principal). This open problem is equivalent to another: does every countable amenable group enjoy the so-called \emph{comparison property} (we refer the reader to \cite{DownarowiczZhang} for details). This equivalence explains the last relaxation in the theorem. It has been proved as \cite[Theorem 6.33]{DownarowiczZhang} that every subexponential group enjoys the comparison property (note that every subexponential group is amenable). Examples of subexponential groups are: Abelian, nilpotent and
virtually nilpotent groups (they have polynomial growth), and the Grigorchuk group, whose growth is subexponential but superpolynomial. In particular, our theorem implies that any algebraic action by any subexponential group $G$ has finite entropy if and only if it admits a principal symbolic extension. Similar relaxation for residually finite amenable groups follows by a different argument, which can also be found~in~\cite{DownarowiczZhang}.

\section{Preliminaries on entropy}

%\subsection{Tilings of countable amenable groups}

%The Ornstein-Weiss Lemma \cite{OW} about subadditivite functions plays a crucial role in the study of entropy theory for actions of countable amenable groups. The examples of applications can be found in \cite{DZ, HYZ1, RW, WZ, We}. The following version this lemma is taken from \cite[1.3.1]{Gronew}.

%Recall that a function $f: \mathfrak{F}_G\rightarrow \mathbb{R}$ (with $f (\emptyset)= 0$ by convention) is \emph{$G$-invariant} if $f (E g)= f (E)$ for all $E\in \mathfrak{F}_G$ and $g\in G$, and it is \emph{subadditive} if $f (E\cup F)\le f (E)+ f (F)$ for all $E, F\in \mathfrak{F}_G$.

%\begin{prop} \label{ow-prop-convergence}
%Let $f: \mathfrak{F}_G\rightarrow \mathbb{R}$ be a nonnegative, $G$-invariant, subadditive function. Then for any F\o lner sequence $\{F_n: n\in \mathbb{N}\}$ of $G$ the sequence $\left\{\frac{f (F_n)}{|F_n|}: n\in \mathbb{N}\right\}$ converges and the value of the limit is independent of the selection of $\{F_n: n\in \mathbb{N}\}$.
%\end{prop}

%The following global version of the Ornstein-Weiss quasitiling theorem was proved as \cite[Theorem 5.2]{DHZ}.

%\begin{theorem} \label{DHZ}
%There exists a F\o lner sequence $F_1, \cdots, F_{n_1}, F_{n_1+ 1}, \cdots, F_{n_2}, F_{n_2+ 1}, \cdots$ of $G$, containing the unit $e$, such that each $F_{n_{i+ 1}+ j}, i\in \mathbb{Z}_+, j= 1, \cdots, n_{i+2}$ can be written as a disjoint union of shifts of $F_{n_i+ 1}, \cdots, F_{n_{i+ 1}}$, where by convention $n_0= 0$.
%\end{theorem}

%\subsection{Topological and tail entropy for actions of countable amenable groups}\
Let $(X, G)$ be a topological action. We denote by $\mathcal {B}_X$ the collection of all Borel subsets of $X$. By a {\it cover} of $X$ we mean a family of Borel sets
whose union is $X$. A {\it partition} of $X$ is a cover whose elements are pairwise disjoint. Let us denote the set of all covers by $\GC_X$, the set of all finite open covers by $\GC^o_X$, the set of all finite closed covers by $\GC^c_X$, and the set of all finite partitions by $\GP_X$. If $X$ is totally disconnected, then $\GP_X\cap\GC^o_X\cap\GC^c_X$ is nonempty and its members are called \emph{clopen partitions}. Note that for any $\mathcal U\in\GC_X$ and $g\in G$, the family 
$$
g^{-1}(\mathcal U)=\{g^{-1}(U):U\in\mathcal U\}
$$ 
is also a cover, and, by continuity of the action, if $\mathcal U$ is either open, closed, or a partition, so is $g^{-1}(\mathcal U)$. Let $\mathcal U,\CW\in \GC_X$. The family
$$
\mathcal U\vee \CW= \{U\cap W: U\in \mathcal U, W\in \CW\}
$$
is a cover called the \emph{join} of $\mathcal U$ and $\CW$. Obviously, the families $\GP_X$, $\GC^o_X$ and $\GC^c_X$ are closed under the join operation. The definition of the join extends naturally to any finite or even countable collection of covers, however a countable join of finite covers is no longer finite. Note the obvious fact, that a countable join of closed covers is a closed cover (which can be uncountable). 
For $\CU\in \GC_X$ and $F\subset G$ ($F$ finite or countable), we set
$$
\CU^F= \bigvee_{g\in F} g^{- 1} \CU.
$$
By convention we also let $\CU^\emptyset$ be the trivial cover $\{X\}$.

If each element of a cover $\mathcal U$ is contained in some element of another cover $\CW$ then we say that $\mathcal U$ is \emph{finer than $\CW$} (and write $\mathcal U\succeq \CW$). 

For two covers $\CU,\CW\in \GC_X$ we let 
$$
N(\CU,\CW) = \max_{W\in\CW}\ \min\{|\CU'|:\CU'\subset\CU, W\subset \bigcup\CU'\}
$$
(the minimal integer $N$ such that every element of $\CW$ can be covered by $N$ elements of $\CU$), and we set the \emph{conditional counting entropy of $\CU$ given $\CW$} to be
$$
H(\CU|\CW)=\log N(\CU,\CW).
$$

\subsection{Topological notions}
We shall now recall the key notions of \tl\ entropy and \tl\ tail entropy of a \tl\ $G$-action. They are straightforward adaptations of the Adler--Conheim--McAndrew notion of \tl\ entropy and Misiurewicz' \emph{\tl\ conditional entropy} \cite{Mi1}, respectively (the term \emph{tail entropy} was introduced later). We remark that \tl\ tail entropy for actions of countable amenable groups has been addressed in \cite{ChungZhang, ZhouZhangChen} and also several equivalent definitions are given in \cite{Zhou}.

Let $\CU,\CW\in\GC^o_X$ and on $\GF_G$ consider the nonnegative function $F\mapsto H(\CU^F|\CW^F)$. It is crucial, that the set $F$ appears here in both ``exponents''. Due to this, it is relatively easy to check that this function is $G$-invariant and subadditive, i.e. 
$$
H(\CU^{F_1\cup F_2}|\CW^{F_1\cup F_2})\le H(\CU^{F_1}|\CW^{F_1})+ H(\CU^{F_2}|\CW^{F_2}),\ \ \ \ (F_1, F_2\in \mathfrak{F}_G)
$$ 
(for details see \cite[Lemma 5.2]{ChungZhang}). We define the \emph{topological conditional entropy of the cover $\CU$ given the cover $\CW$} as the following limit
$$
h_G(\CU| \CW) = \lim_{n\rightarrow \infty} \frac{1}{|F_n|} H(\CU^{F_n}|\CW^{F_n}),
$$
where $(F_n)_{n\in\N}$ is a F\o lner \sq\ in $G$. The Ornstein--Weiss Lemma (see e.g. \cite[1.3.1]{Gronew}) guarantees that the above limit exists and does not depend on the choice of the F\o lner sequence. Next, one defines a series of related notions:
\begin{align*}
&h_G(\CU)=h_G(\CU|\{X\}),\\
&h_G(X|\CW)= \sup_{\CU\in \GC_X^o} h_G(\CU| \CW),\\
&h_G(X) = \htop(X,G) = h_G(X|\{X\})=\sup_{\CU\in \GC_X^o} h_G(\CU),\\
&h^*(X,G)= \inf_{\CW\in \GC_X^o} h_G(X|\CW).
\end{align*}
They are called, respectively, \emph{\tl\ entropy of the cover $\CU$}, 
\emph{\tl\ conditional entropy (of the action) given the cover $\CW$},
\emph{\tl\ entropy (of the action)}, and \emph{\tl\ tail entropy (of the action)}.

Directly by the definitions, we have $h^*(X,G)\le\htop(X,G)$. It is also not hard to see that if $\htop(X,G)=\infty$ then also $h^*(X,G)=\infty$. Following \cite{Mi1}, we introduce another key notion of this paper. We remark that the adaptation to amenable group actions has already been addressed~in~\cite{ChungZhang}.
\begin{definition}
A \tl\ action $(X,G)$ is \emph{\ahe} if $h^*(X,G)=0$.
\end{definition}

\smallskip

We now review a family of notions very similar to the previous ones, in which we replace the conditioning cover $\CW^{F_n}$ by the invariant (and usually uncountable) cover $\CW^G$. In order to distinguish the two approaches, we will now use the term ``relative'' rather than ``conditional''. 
\begin{definition}
Let $(F_n)_{n\in \mathbb{N}}$ be a F\o lner sequence of $G$ and let $\CU\in\GC_X^o$ while $\CW$ is any cover of $X$. The \emph{\tl\ relative entropy} of the cover $\CU$ given the cover $\CW$ is the limit
$$
\bar h_G(\CU|\CW)= \lim_{n\rightarrow \infty}\frac{1}{|F_n|}H(\CU^{F_n}|\CW^G).
$$
\end{definition}
The function $F\mapsto H(\CU^F|\CW^G)$, is obviously subadditive and \inv\ on $\GF_G$. Thus, by the already mentioned Ornstein--Weiss Lemma, the limit in the definition exists and does not depend on the F\o lner \sq. 

A special case of the above relative entropy occurs when the cover $\CW$ is a closed \usc\ partition\footnote{A closed partition is \usc\ if and only if it defines a closed equivalence relation.}. Then $\CW^G$ defines a closed invariant equivalence relation in $X$. It is well known that such a relation determines a \tl\ factor of the system $(X,G)$, say $(Z,G)$. In this case we will denote $\bar h_G(\CU|\CW)$ by $\bar h_G(\CU|Z)$ and call it the \emph{relative entropy of the cover $\CU$ given the factor $(Z,G)$}. Conversely, given a \tl\ factor $(Z,G)$ of $(X,G)$ (via a \tl\ factor map $\pi:X\to Z$), the partition $\CW$ of $X$ into the fibers of $\pi$ is closed, \usc\ and invariant (i.e.\ $\CW^G=\CW$). In this case, we have
$$
\bar h_G(\CU|Z)=\bar h_G(\CU|\CW).
$$
Finally, given a \tl\ factor $(Z,G)$ of $(X,G)$ we define the \emph{\tl\ relative entropy of the system $(X,G)$ given the factor $(Z,G)$}, as
$$
\bar h_G(X|Z)=\sup_{\CU\in\mathfrak C^o_X} \bar h_G(\CU|Z).
$$

It is elementary to see that if $(X,G)$ and $(Y,G)$ are subshifts and $\CU$ and $\CW$ are the clopen ``zero-coordinate'' partitions (or ``one-symbol'' partitions), then 
$$
\bar h_G(X|Z)=\bar h_G(\CU|\CW).
$$

\subsection{Measure-theoretic notions}
We now pass to the review of measure-theoretic notions of entropy. We consider a measure-preserving $G$-action $(X,\Sigma_X,\mu,G)$, that is, $(X,\Sigma_X,\mu)$ is a probability space on which $G$ acts by measure-preserving automorphisms. For example, if $(X,G)$ is a \tl\ action and $\mu\in\M(X,G)$ then $(X,\mathcal B_X,\mu,G)$ is a measure-preserving $G$-action. Let $\P\in \GP_X$. The \emph{Shannon entropy} of $\P$ equals
$$
H_\mu(\P)=-\sum_{P\in\P}\mu(P)\log(\mu(P)).
$$
Shannon entropy satisfies \emph{strong subadditivity}, i.e.\
$$
H_\mu(\P^{F_1\cup F_2})\le H_\mu(\P^{F_1})+H_\mu(\P^{F_2})-H_\mu(\P^{F_1\cap F_2}),\ \ \ \ (F_1, F_2\in \mathfrak{F}_G).
$$
The \emph{dynamical entropy of $\P$ with respect to $\mu$} is defined as 
$$
h_\mu(\P,G) = \lim_n \frac1{|F_n|}H_\mu(\P^{F_n})= \inf_{F\in\GF_G} \frac1{|F|}H_\mu(\P^F),
$$
where $(F_n)_{n\in\N}$ is a F\o lner \sq\ in $G$, while the latter equality is derived using strong subadditivity (the derivation can be found e.g.\ in \cite{DFR}). In particular, the dynamical entropy of a partition (and hence also the Kolmogorov--Sinai entropy defined below) does not depend on the choice of the F\o lner \sq.
The \emph{Kolmogorov--Sinai entropy} of the measure-preserving $G$-action $(X,\Sigma_X,\mu,G)$ is defined as
$$
h_\mu(X,G)=\sup_{\P\in\GP_X} h_\mu(\P,G).
$$

Next, consider two partitions $\P,\Q\in\GP_X$. One of the simplest ways of defining the \emph{conditional Shannon entropy of $\P$ given $\Q$} is
$$
H_\mu(\P|\Q) = H_\mu(\P\vee\Q)- H_\mu(\Q).
$$
For a partition $\P\in\GP_X$ and a sub-sigma-algebra $\Sigma\subset\Sigma_X$, the \emph{conditional Shannon entropy of $\P$ given $\Sigma$} can be defined (in one of several equivalent ways) as
$$
H_\mu(\P|\Sigma) = \inf_{\Q\in\GP_X, \Q\subset\Sigma} H_\mu(\P|\Q).
$$

Now, going back to two partitions $\P,\Q\in\GP_X$, one defines the \emph{conditional entropy of the process generated by $\P$ given the process generated by $\Q$} as
$$
h_\mu(\P,G|\Q)= \lim_n \frac1{|F_n|}H_\mu(\P^{F_n}|\Q^{F_n}).
$$
We point out that the conditional entropy $H_\mu(\P^{F_n}|\Q^{F_n})$ viewed as a function of $\mathfrak F_G$ is subadditive, {\bf but not} strongly subadditive (see Example \ref{example}).  Thus, although the existence and independence on the F\o lner \sq\ of the above limit is guaranteed by the aforementioned Ornstein--Weiss Lemma, it does not follow directly that the above limit equals the corresponding infimum over $\mathfrak F_G$. It actually is, but the proof requires a stronger tool -- the Pinsker formula (see below).

On the other hand, the function $H_\mu(\P^{F_n}|\Q^G)$, where $\Q^G$ is the {\bf\inv} sub-sigma-algebra generated by all partitions $g^{-1}(\Q)$, $g\in G$, is strongly subadditive, which allows one to define the \emph{relative entropy of the process generated by $\P$ given the process generated by $\Q$}, as
$$
\bar h_\mu(\P,G|\Q)= \lim_n \frac1{|F_n|}H_\mu(\P^{F_n}|\Q^G)=\inf_{F\in\mathfrak F_G}
\frac1{|F|}H_\mu(\P^{F}|\Q^G).
$$

The Pinsker formula is responsible for the equality between conditional and relative entropies, as follows:
$$
h_\mu(\P,G|\Q)= \bar h_\mu(\P,G|\Q)
$$
(this version of the Pinsker formula for $G$-actions can be found as \cite[Theorem 4.4]{WZ} and also as \cite[Lemma 1.1]{GTW}). This is the reason, why we can forget about distinction between these two measure-theoretic notions.

\smallskip
Now consider a measure-theoretic factor map $\pi:Y\to X$ between two measure-preserving $G$-actions $(Y,\Sigma_Y,\nu,G)$ and $(X,\Sigma_X,\mu,G)$ (equivalently, one can consider just an \inv\ sub-sigma-algebra $\Sigma_X\subset\Sigma_Y$ and $\mu$ equal to the restriction of $\nu$ to $\Sigma_X$). In this case we define the \emph{conditional entropy of $\nu$ given $X$} (equivalently, given $\Sigma_X$), as follows
\begin{equation}\label{eq}
h_\nu(Y,G|X)= \inf_{\Q\in\GP_X}\ \sup_{\P\in\GP_Y} h_\nu(\P,G|\Q),
\end{equation}
where $\P$ ranges over all finite measurable partitions of $Y$, while $\Q$ ranges over all finite measurable partitions of $X$ lifted to $Y$ (equivalently, $\Q\subset\Sigma_X$). If $h_\mu(X,G)<\infty$ then $h_\nu(Y,G|X)$ is simply the difference $h_\nu(Y,G)-h_\mu(X,G)$. 
Because of this ``difference formula'', the infimum over $\Q$ and supremum over $\P$ in \eqref{eq} can be swapped.
\smallskip

The notion provided below was coined by Ledrappier \cite{Ledrappier} (he did it for $\Z$-actions, but the extension to $G$-actions requires no modification).

\begin{definition}
Let $\pi:Y\to X$ be a \tl\ factor map between two \tl\ actions $(Y,G)$ and $(X,G)$. If, for every $\nu\in\M(Y,G)$, $h_\nu(Y,G|X)=0$ then $(Y,G)$ (together with the factor map $\pi$) is called a \emph{principal extension of $(X,G)$}.
\end{definition}

By the standard variational principle, it follows that principal extension preserves \tl\ entropy. On the other hand, preservation of \tl\ entropy is a strictly weaker condition.

\subsection{Entropy structure}
By a \emph{structure} on a compact metric space $\M$ we will understand any nondecreasing \sq\ of commonly bounded nonnegative functions on $\M$, $\F =(f_k)_{k\ge 0}$ with $f_0\equiv 0$. Clearly, the pointwise \emph{limit function} $f=\lim_k f_k$ exists and is nonnegative and bounded.
Two structures $\F =(f_k)_{k\ge 0}$ and $\F' = (f'_k)_{k\ge 0}$ are said to be \emph{uniformly equivalent} if
$$
\forall_{\varepsilon>0,\,k_0\ge 0}\ \exists_{k\ge 0}\ \ (f'_k>f_{k_0}-\varepsilon \text{ and } f_k>f'_{k_0}-\varepsilon).
$$
Notice the obvious fact that uniformly equivalent structures have a common limit function.

Entropy structure of a topological $G$-action is defined in \cite{DownarowiczZhang} in exactly the same manner as it is done for $G=\Z$ in \cite{BD}. We present here only one of several equivalent definitions. By \cite[Theorem 3.2]{H}, any topological action $(X, G)$ has a principal \zd\ extension $(X', G)$. This allows one to define entropy structure in two steps, the first one being the definition for actions on \zd\ spaces, while the second is an easy generalization relying on principal \zd\ extensions. 

\begin{definition}\label{ensz}
Let $(X,G)$ be a topological action on a \zd\ space and assume that $\htop(X,G)<\infty$.
Let $(\P_k)_{k\in\N}$ be a \emph{refining} \sq\ of finite clopen partitions, meaning that $\P_{k+1}\succcurlyeq\P_k$ for each $k\in\N$ and $\lim_k\diam(\P_k)=0$, where $\diam(\P_k)$ stands for the maximal diameter of an atom of $\P_k$. For every $\mu\in\M(X,G)$ we set $h_k(\mu) = h_\mu (\P_k, G)$. Then the \sq\ $\H = (h_k)_{k\ge 0}$ on $\M(X,G)$ is a structure called an \emph{\ens} of $(X,G)$.
\end{definition}

It is obvious that the nondecreasing limit $\lim_{k}h_k$ equals the entropy function $h$ on $\M(X,G)$, given by $h(\mu)=h_\mu(X,G)$. It is known that \ens s arising from different refining \sq s of finite clopen partitions are uniformly equivalent (see \cite{BD} or \cite{D2005})). Thus, \ens\ can be defined as the corresponding uniform equivalence class and in this understanding it does not depend on the refining \sq\ of finite clopen partitions (which is equivalent to saying that it is a conjugacy invariant within the class of \zd\ systems).

\begin{definition}\label{ensf}
	Let $(X,G)$ be a topological action with finite entropy. We define the \emph{\ens} of $(X,G)$ as any structure $\H=(h_k)_{k\ge 0}$ on $\M(X,G)$, such that for any principal \zd\ extension $\pi':X'\to X$, and any entropy structure $\H'=(h'_k)_{k\ge 0}$ on $\M(X',G)$, the structure $\H=(h_k)_{k\ge 0}$ lifted against $\pi'$ (i.e.\ the \sq\ $(h_k\circ\pi')_{k\in\N}$) is uniformly equivalent to $\H'$.
\end{definition}

The existence of such an entropy structure (which is by no means obvious) is proved for $\Z$-actions in \cite{BD} (see also \cite{D2005}), while the adaptation to actions of amenable groups can be found in \cite{DownarowiczZhang}. 

It is clear that properly understood entropy structure is a uniform equivalence class and in this understanding it is a conjugacy invariant. Nonetheless, it is technically much more convenient to work with entropy structures understood as individual \sq s of functions $(h_k)_{k\ge 0}$. 
It was proved in \cite{DownarowiczZhang} that an \ens\ of a \zd\ topological $G$-action consists of affine functions $h_k$ such that $h_k-h_{k-1}$ is \usc\ for every $k\in\N$ (in particular, each $h_k$ is \usc). In such case we will say that $\H$ is an \emph{affine structure with \usc\ differences}. In \cite{DownarowiczZhang} (following \cite{BD} and \cite{D2005}) it is shown that any topological action $(X, G)$ with finite entropy has an entropy structure which is affine and has \usc\ differences.

For any bounded from above function $f$ defined on a compact metric space $\mathcal{Mi1}$ we denote by $\tilde{f}$ the \emph{upper semicontinuous envelope} of $f$, i.e.\ the function
$$
\tilde{f} (x) = \max \{f (x), \limsup\limits_{x'\rightarrow x} f (x')\}.
$$
In fact, $\tilde{f}$ equals the pointwise infimum of all continuous functions dominating $f$ and it is the smallest upper semicontinuous function dominating $f$.

Now let $\H=(h_k)_{k\ge 0}$ be an entropy structure of a topological action $(X,G)$ with finite entropy and let $h$ be the entropy function on $\M(X,G)$. We set
$$\theta_k (\mu) = h (\mu) - h_k (\mu)\ \ \text{and}\ \ u_1 (\mu) = \lim_{k\rightarrow \infty} \tilde{\theta_k} (\mu)\ \ \ (\mu\in \mathcal{Mi1} (X, G)).$$ 

Note that the functions $\theta_k$ are equal to the conditional entropies $h_\mu(X,G|X_k)$, where $(X_k,G)$ is the symbolic factor $(X,G)$ generated by the clopen partition $\P_k$, and they converge pointwise to zero. So, $u_1$ relies only on the defects of upper semicontinuity of the conditional entropy functions $\theta_k$. The function $u_1$ is sometimes called the \emph{tail entropy function} (on \im s), although this term is also used for a slightly different function. The reason why it is denoted by $u_1$ is that this is the first of a transfinite \sq\ of functions $(u_\alpha)_{\alpha<\omega_1}$ responsible for ``higher order defects''. For more details the reader is referred to either \cite{BD}, \cite{D2005} or \cite{D1} in case of $\Z$-actions, or to \cite{DownarowiczZhang} for $G$-actions (as a matter of fact, the acting countable amenable group plays no role in this context).

\section{The \tl\ Pinsker formula}
The aim of this section is to prove a ``\tl\ Pinsker formula'', i.e.\ equality between \tl\ conditional and \tl\ relative entropies. This formula will be needed in the next section, in the proof of the conditional variational principle for \zd\ systems, which in turn will allow us to prove the tail variational principle for such actions. We begin with a lemma.

\begin{lemma} \label{1306140805}
Let $\CU\in \GC_X^o$ and $\{\mathcal{V}_n: n\in \mathbb{N}\}\subset \GC_X^c$. Let $\CW=\bigvee_{n\in\N}\mathcal V_n$ and, for each $n\in\N$, let $\CW_n=\bigvee_{i=1}^n\mathcal V_i$. Then there exists $n_0\in \mathbb{N}$ such that,
for each $n\ge n_0$, we have
$$
N(\CU,\CW)=N(\CU,\CW_n).
$$
\end{lemma}
\begin{proof}
Consider the compact metric space $\mathbf{V}=\prod_{n\in \mathbb{N}} \mathcal{V}_n$, where each $\mathcal{V}_n$ is viewed as a discrete space containing exactly $|\mathcal{V}_n|$ points. 
This space consists of formal \sq s $\mathbf v=(V_n)_{n\in\N}$ such that $V_n\in\mathcal V_n$ for each $n\in\N$, and for every such element $\mathbf v$ we let $W_n(\mathbf v)= \bigcap_{i=1}^n V_i$ and $W(\mathbf v)= \bigcap_{n\in\N} V_n$. Clearly the families
$\{W_n(\mathbf v):\mathbf v\in\mathbf V\}$ and $\{W(\mathbf v):\mathbf v\in\mathbf V\}$ coincide with $\CW_n$ and $\CW$, respectively. On $\mathbf V$ we define the following functions
\begin{align*}
&f_n(\mathbf v) = \min\{|\CU'|:\CU'\subset\CU, W_n(\mathbf v)\subset\bigcup\CU'\}, \ \ (n\in\N),\\
&f_\infty(\mathbf v) = \min\{|\CU'|:\CU'\subset\CU, W(\mathbf v)\subset\bigcup\CU'\},
\end{align*}
and then we have the following obvious identities:
\begin{align*}
&N(\CU,\CW_n)=\max\{f_n(\mathbf v):\mathbf v\in\mathbf V\}, \ \ (n\in\N),\\ 
&N(\CU,\CW) = \max\{f_\infty(\mathbf v):\mathbf v\in\mathbf V\}. 
\end{align*}
Clearly, the \sq\ of functions $(f_n)_{n\in\N}$ is nonincreasing and dominates $f_{\infty}$. Further, the functions $f_n$ are by definition constant on the clopen cylinders in $\mathbf V$ consisting of \sq s $\mathbf v$ which agree on the initial $n$ positions. This implies that the functions $f_n$ are continuous. Once we show that the functions $f_n$ converge pointwise to $f_\infty$, the proof will be ended, because in such case, by the already mentioned elementary ``exchange of suprema and infima'' statement (see e.g.\ \cite[Proposition 2.4 ]{BD}), we will have 
$$
\max\{f_\infty(\mathbf v):\mathbf v\in\mathbf V\}=\max_{\mathbf v\in\mathbf V}\ \inf_n f_n(\mathbf v) = \inf_n\ \max_{\mathbf v\in\mathbf V}f_n(\mathbf v).
$$
Since the values of $f_n$ are integers, the latter infimum must be attained for some $n_0$.

The pointwise convergence of the functions $f_n$ to $f_\infty$ is proved as follows. For $\mathbf v\in\mathbf V$, let $\CU'$ be a subfamily of $\CU$ which covers $W(\mathbf v)$, and assume it has the minimal cardinality $f_\infty (\mathbf v)$. Since the sets
$W_n(\mathbf v)$ are compact and their decreasing intersection is $W(\mathbf v)$, it is clear that for some $n$, $W_n(\mathbf v)$ is covered by the same (open) family $\CU'$. This shows that $f_n(\mathbf v) \le f_\infty(\mathbf v)$ for some $n$, while the converse inequality is trivial. 
\end{proof}

We are able to prove the \tl\ Pinsker formula in case $\CU$ is open and $\CW$ is closed:

\begin{theorem} \label{1007031233}
Let $(F_n)_{n\in \mathbb{N}}$ be a F\o lner sequence of $G$ and $\CU\in \GC_X^o,\CW\in\GC_X^c$. Then
$$
h_G(\CU|\CW)= \bar h_G(\CU|\CW).
$$
\end{theorem}

\begin{remark}
For $\mathbb{Z}$-actions and the standard F\o lner \sq\ $F_n= \{0,1,\dots,n-1\}$ the above equality follows quickly from Lemma \ref{1306140805}. The key ingredient is that both conditional counting entropies, $H(\CU^{F}|\CW^{F})$ and $H(\CU^{F}|\CW^G)$, are subadditive on $\mathfrak F$, which for $\Z$-actions is sufficient for the ``infimum rule'' to hold: the involved limits on both sides are in fact infima. For more general $G$-actions, subadditivity is insufficient for the infimum rule; strong subadditivity (or at least Shearer's inequality, see e.g.\ \cite{DFR}) is required. The condititonal counting entropy is not strongly subadditive (even for clopen partitions, see Example \ref{example}) and thus the proof is more complicated.
\end{remark}

\begin{proof}
Since $N(\CU^F,\CW^F)\ge N(\CU^F,\CW^G)$ for each $F\in \mathfrak{F}_G$, the inequality ``$\ge$'' in the theorem is obvious and it remains to show that for a suitable F\o lner \sq\
$(F_n)_{n\in\N}$, every $\eps>0$ and large enough $n$, one has
\begin{equation} \label{1307082008}
\frac{1}{|F_n|}H(\CU^{F_n}|\CW^{F_n})\le\bar h_G(\CU|\CW)+2\eps.
\end{equation}

We now invoke the theory of tilings\footnote{The quasitilings introduced by Ornstein and Weiss in \cite{OW} should also suffice, but tilings are simply more convenient to use.}. It follows from \cite[Theorem 5.2]{DHZ} that there exists a F\o lner \sq\ $(F_n)_{n\in\N}$ of sets containing the unit $e$, which breaks into countably many portions $\{F_1,F_2,\dots,F_{n_1}\}$, $\{F_{n_1+1},F_{n_1+2},\dots,F_{n_2}\}$, $\{F_{n_2+1},F_{n_2+2},\dots,F_{n_3}\}$, $\dots$, such that every set $F_n$ with $n>n_{i+1}$ is a disjoint union of shifted sets from the $i$th portion:
\begin{equation} \label{tiling}
F_n=\bigcupdot_{k=n_i+1}^{n_{i+1}}\ \bigcupdot_{c\in C_{k,n}}F_kc
\end{equation}
(that is, such finite sets $C_{k,n}\subset F_n$ exist). Now, fix any $\eps>0$. There exists an $i\in \mathbb{N}$ such that
$$
\frac{1}{|F_k|} H(\CU^{F_k}|\CW^G)< \bar h_G(\CU|\CW)+\eps, \text{ \ for each\ }\ k\in \{n_i+1,\dots,n_{i+1}\}.
$$
For each $k$ as above we apply Lemma \ref{1306140805} to $\CU^{F_k}$ and get a set $B_k\in \mathfrak{F}_G$ (we can assume that $e\in B_k$), such that $N(\CU^{F_k},\CW^G)=N(\CU^{F_k},\CW^{B_k})$. 
Let $B=\bigcup_{k=n_i+1}^{n_{i+1}}B_k$ (clearly, $B\in \mathfrak{F}_G$). Then we have
$$
N(\CU^{F_k},\CW^G)= N(\CU^{F_k},\CW^B),
$$
for all $k\in \{n_i+1,\dots,n_{i+1}\}$, in particular
\begin{equation}\label{tera}
H(\CU^{F_k}|\CW^B)<|F_k|(\bar h_G(\CU|\CW)+\eps).
\end{equation}
For any $\delta>0$ and large enough $n\in\N$, the set $F_n$ (henceforth abbreviated as $F$) is $(B,\delta)$-invariant (i.e. $\frac{|B F_n\Delta F_n|}{|F_n|}< \delta$ where $\Delta$ denotes the symmetric difference of sets), and it splits as the disjoint union \eqref{tiling}.

Let $F_B$ denote the \emph{$B$-core} of $F$, that is $F_B=\{g\in F: Bg\subset F\}$. For $\delta$ sufficiently small, the $(B,\delta)$-invariance of $F$ implies that 
\begin{equation}\label{delta}
|F\setminus F_B|<\frac{\eps}{D\cdot (1+ \log|\CU|)}|F|, 
\end{equation}
where $D=(n_{i+1}-n_i)\max\{|F_k|:k=n_i+1,\dots,n_{i+1}\}$. We have, by subadditivity
\begin{equation} \label{1007022115}
\frac{1}{|F|}H(\CU^{F}|\CW^{F})
\le \frac1{|F|}\sum_{k=n_i+1}^{n_{i+1}}\ \sum_{c\in C_{k,n}}H(\CU^{F_kc}|\CW^F).
\end{equation}
For each $k\in\{n_i+1,\dots,n_{i+1}\}$ we have (using \eqref{tera} and \eqref{delta} to get the last inequality):
\begin{multline*}
\sum_{c\in C_{k,n}}H(\CU^{F_kc}|\CW^F) =
\sum_{c\in C_{k,n}\setminus F_B}H(\CU^{F_k}|\CW^{Fc^{-1}})+
\sum_{c\in C_{k,n}\cap F_B}H(\CU^{F_k}|\CW^{Fc^{-1}})\\
\le |F\setminus F_B|\cdot |F_k|\cdot \log|\CU|+
\sum_{c\in C_{k,n}\cap F_B}H(\CU^{F_k}|\CW^{B})
\ \ \le \ \ \frac\eps{n_{i+1}-n_i}|F|+|C_{k,n}|\cdot |F_k|(\bar h_G(\CU|\CW)+\eps).
\end{multline*}
Therefore, the right hand side (and thus also the left hand side) of \eqref{1007022115} is dominated by
$$
\frac1{|F|}\sum_{k=n_i+1}^{n_{i+1}} \Bigl(\frac\eps{n_{i+1}-n_i}|F|+|C_{k,n}|\cdot |F_k|(\bar h_G(\CU|\CW)+\eps)\Bigr) = \bar h_G(\CU|\CW)+2\eps,
$$
(we have used the equality $\sum_{k=n_i+1}^{n_{i+1}}|C_{k,n}|\cdot |F_k|=|F|$, which follows from the disjoint union \eqref{tiling}). We have just shown \eqref{1307082008}, which ends the proof.
\end{proof}

\section{Conditional variational principle for \zd\ systems}

Our main Theorems \ref{tail} and \ref{main result} depend on the following \emph{variational principle for topological conditional entropy in \zd\ systems}. By the \tl\ Pinsker formula (Theorem~\ref{1007031233}), the topological conditional entropy of a clopen disjoint cover given another clopen disjoint cover equals the topological relative entropy given a \tl\ factor. From here we can use some existing results.

\begin{prop} \label{201505172121}
	Let $(X, G)$ be a topological action and $\mathcal{P}, \mathcal{Q}\in \GP_X$. Assume that $\mathcal{P}\succeq \mathcal{Q}$, and that $\mathcal{P}$ (and hence $\mathcal{Q}$) is clopen (i.e. $\mathcal{P}, \mathcal{Q}\in \GP_X\cap\GC^o_X\cap\GC^c_X$). Then
	\begin{equation*}
		h_G (\P|\Q)= \max \{h_\mu(\P,G)- h_\mu(\Q,G): \mu\in\M(X,G)\}.
	\end{equation*}
\end{prop}
\begin{proof}
Since $\mathcal{P}$ is a finite clopen partition of $X$, it generates naturally 
a symbolic \tl\ factor $(X_\P,G)$ of $(X,G)$ with the alphabet $\P$ via the \emph{itinerary map}
$$
\pi_\P(x) = (x_g)_{g\in G},
$$
where, for each $g\in G$, we set $x_g = P\in\P \iff g(x)\in P$. The action of $G$ on such itineraries is the standard shift: $h((x_g)_{g\in G})=(x_{gh})_{g\in G}$\ \ $(h\in G)$. Likewise, $\Q$ generates a symbolic $G$-action $(X_\Q,G)$ with the alphabet $\Q$ and there is a natural \tl\ factor map from $(X_\P,G)$ onto $(X_\Q,G)$ which coincides with the following \emph{one-block code} $\pi_{\P,\Q}$:
$$
\pi_{\P,\Q}((x_g)_{g\in G})=(y_g)_{g\in G},
$$
where, for each $g\in G$, $y_g$ is the unique element of $\Q$ which contains $x_g$.

Now, since $\P$ is (in particular) an open cover and $\Q$ is (in particular) a closed cover, we can apply our \tl\ Pinsker formula (Theorem \ref{1007031233}), which reads
$$
h_G(\P|\Q)=\bar h_G(\P|\Q).
$$
Because $\P$ and $\Q$ are clopen partitions and they are \tl\ generators of the systems $(X_\P,G)$ and $(X_\Q,G)$, respectively, the right hand side equals the topological relative entropy $\bar h_G(X_\P|X_\Q)$ of the topological factor map $\pi_{\P,\Q}$ for which the variational principle is already known, see \cite[Theorem 13.3]{DZ} (another proof can be found in \cite[Theorem 5.1]{Yan}):
$$
\bar h_G(X_\P|X_\Q) = \sup\{h_\nu (X_\P,G|X_\Q):\nu\in\M(X_\P,G)\}.
$$
Since $\P$ and $\Q$ are measure-theoretic generators in $(X_\P,G)$ and $(X_\Q,G)$ for any \im\ in the respective systems, we have, for each $\nu\in\M(X_\P,G)$, the equality
$$
h_\nu (X_\P,G|X_\Q)=h_\mu(\P,G|\Q)=h_\mu(\P,G)- h_\mu(\Q,G),
$$
where $\mu$ is any \im\ on $X$ satisfying $\nu = \pi_\P(\mu)$.
Thus, we obtain the desired equality
$$
h_G(\P|\Q)=\max \{h_\mu(\P,G)- h_\mu(\Q,G): \mu\in \mathcal{M} (X, G)\},$$
where the supremum is replaced by the maximum due to upper semicontinuity of the function
$\mu\mapsto h_\mu(\P,G|\Q)$ for finite clopen partitions.
\end{proof}

\section{Preservation of \tl\ tail entropy}
In this section we show that topological tail entropy of a \tl\ $G$-action is preserved by principal extensions. This fact for $\mathbb{Z}$-actions is well known and it was proved by F.\ Ledrappier, see \cite[Theorem 3]{Ledrappier}. It also follows \emph{a posteriori} from the tail variational principle and the (discussed earlier) preservation of entropy structure by principal extensions. But since we intend to use it in the derivation of the tail variational principle, we need an independent proof. The shortest proof of Proposition \ref{principal} that we were able to come up with relies heavily on the already mentioned variational principle for topological relative entropy (\cite[Theorem 13.3]{DZ} or \cite[Theorem 5.1]{Yan}) and, in the most crucial place, on a result from \cite{ZhouZhangChen}. We use these results to jump several times between \tl\ and measure-theoretic notions. Clearly, we do not claim any credit for this Proposition, and also we allow ourselves to be a bit sketchy.

\begin{prop} \label{principal}
Let $\pi:Y\to X$ be a \tl\ factor map between two \tl\ actions $(Y,G)$ and $(X,G)$. If $\pi$ is a principal extension then $h^* (Y, G)= h^* (X, G)$.
\end{prop}

\begin{proof}
Since $\pi$ is a principal extension, it preserves \tl\ entropy. As we have already remarked,
infinite \tl\ entropy implies infinite \tl\ tail entropy for both actions. It remains to continue assuming that $\htop(X,G)<\infty$. 

Consider the product action $(X\times X, G)$ (given by $g(x_1,x_2)=(g(x_1),g(x_2))$), and its projection to the first coordinate $\pi_1: X\times X\rightarrow X$ (which is a topological factor map). Then by \cite[Theorem 3.1]{ZhouZhangChen} one has
\begin{equation} \label{tce-X}
h^*(X, G)= \max\Bigl\{\Bigr(\tilde h_\mu(X\times X,G|X)-h_\mu(X\times X,G|X)\Bigr): \mu\in\M(X\times X,G)\Bigr\},
\end{equation}
and an analogous formula holds for $(Y\times Y,G)$ factoring onto $(Y,G)$. For a real-valued function $f$ on a compact space we call the difference $\tilde f - f$ \emph{the defect function} (meaning the defect of upper semi-continuity). So, in the large round parentheses above we see the defect function of the function $\mu\mapsto h_\mu(X\times X,G|X)$ on $\M(X\times X,G)$. We need to show that these defect functions on $X\times X$ and $Y\times Y$ have the same maxima. Since $\pi:Y\to X$ is a principal extension, the variational principle for \tl\ relative entropy of a topological factor map between topological $G$-actions implies that the \tl\ relative entropy associated with the \tl\ factor map $\pi:Y\to X$ equals zero. Further, the \tl\ relative entropy associated with the \tl\ factor map $\pi\times\pi:Y\times Y\to X\times X$ is at most the sum of two zeros (this can be readily seen as it suffices to consider product covers of $Y\times Y$), so it equals zero as well. Applying the same relative variational principle ``backwards'', we deduce that $\pi\times\pi$ is a principal extension\footnote{It is probably possible to prove this fact directly using only measure-theoretic entropy, but we failed to find a short argument of this kind.}. This implies that the function $\nu\mapsto h_\nu(Y\times Y,G|Y)$ on $\M(Y\times Y,G)$ equals the function $\mu\mapsto h_\mu(X\times X,G|X)$ on $\M(X\times X,G)$ lifted against (i.e.\ composed with) $\pi\times\pi$. It takes an elementary exercise in topology to see that in the context of a continuous surjection between two compact metric spaces, the operation of taking the \usc\ envelope of a function commutes with operation of lifting. The same is obviously true for the defect operation, and this observation ends the proof.
\end{proof}	
	
\section{Proofs of the main theorems}
Now we are ready to prove our main Theorems \ref{tail} and \ref{main result}. For the  reader's convenience we state them again before each proof.

\begin{customthm}{2.1}
Let $(X,G)$ be a topological action of a countable amenable group, with finite entropy. Then
$$
h^* (X,G)= \max \{u_1 (\mu): \mu\in \M(X,G)\}= \lim_{k\rightarrow \infty} 
\ \sup\{\theta_k (\mu):\mu\in\M(X,G)\}.
$$
\end{customthm}

\begin{proof}
As we have already mentioned, the second equality is just swapping the limit of a decreasing \sq\ of upper semicontinuous functions with the maximum over the domain, see e.g.\ \cite[Proposition 2.4 ]{BD} and then applying the obvious fact that once the maximum is inside, we can skip the \usc\ envelope marks (and replace maximum by supremum). It remains to prove that
\begin{equation} \label{xxxx}
	h^* (X,G)= \lim_{k\rightarrow \infty}\ \sup\{\theta_k (\mu):\mu\in\M(X,G)\}.
\end{equation}

At first we assume that $X$ is \zd, and that the entropy structure of $(X,G)$, $\H=(h_k)_{k\ge 0}$, is determined by a refining \sq\ $(\P_k)_{k\in\N}$ of finite clopen partitions. In this case it is easy to see that both the supremum and infimum in the definition of $h^*(G,X)$ are realized along the \sq\ $\P_k$ (viewed as open covers), i.e.\ we can write
\begin{equation} \label{bridge}
h^* (X,G)= \inf_{k\in \mathbb{N}}\ \sup_{m\in \mathbb{N}} h_G (\P_m| \P_k).
\end{equation}
Then, by Theorem \ref{201505172121}, we obtain 
\begin{multline*}
h^* (X,G)= \inf_{k\in\N}\ \sup_{m\in\N}\ \max\{h_m(\mu)-h_k(\mu):\mu\in\M (X,G)\}\\
=\inf_{k\in\N}\ \max\left\{\sup_{m\in\N}(h_m(\mu)-h_k(\mu)):\mu\in\M (X,G)\right\}=
\inf_{k\in\N}\ \max\{\theta_k(\mu):\mu\in\M (X,G)\},
\end{multline*}
and we are done.
	  
Now consider a general topological $G$-action $(X,G)$.  By \cite[Theorem 3.2]{H}, this system has a principal \zd\ extension $(X',G)$. It is fairly obvious that the quantity 
$$
\lim_{k\rightarrow \infty}\ \sup\{\theta_k (\mu):\mu\in\M(X,G)\}
$$
does not depend on the choice of an individual entropy structure within the uniform equivalence class. This implies that the above quantity for $(X,G)$ is the same as for $(X',G)$. On the other hand, by Proposition \ref{principal}, we also have $h^*(X,G)=h^*(X',G)$. Since we have already proved the desired  equality for \zd\ actions, this equality passes to $(X,G)$, and the proof is finished.
\end{proof}

\begin{customthm}{2.2}
Let $(X,G)$ be a \tl\ action of a countable amenable group. Then the following conditions are equivalent:
 \begin{enumerate}
 \item $(X,G)$ is \ahe.
 \item The entropy structure of $(X,G)$ converges uniformly to the entropy function.
% \item The entropy function is the minimal superenvelope of the entropy % structure.
 \item $(X, G)$ admits a principal quasi-symbolic extension.
 \item For any $\eps>0$ the action admits a quasi-symbolic extension with \tl\ relative 
 entropy at most $\eps$.
 \end{enumerate}
Furthermore, if $G$ is either residually finite or enjoys the comparison property then the quasi-symbolic extensions in the above statements can be replaced by symbolic extensions.
\end{customthm}

\begin{proof}
The equivalence of $(1)\Longleftrightarrow (2)$ follows from the tail variational principle Theorem \ref{tail} and the obvious fact that the entropy structure converges uniformly to the entropy function if and only if the \sq\ $(\theta_k)_{k\in\N}$ converges uniformly to zero. The equivalence of $(2)\Longleftrightarrow (3)$ is just \cite[Theorem 5.6]{DownarowiczZhang}\footnote{We remark that in \cite{DownarowiczZhang} asymptotic $h$-expansiveness of a topological $G$-action was defined as uniform convergence of the entropy structure.}. The equivalence of $(3)\Longleftrightarrow (4)$ follows from the symbolic extension theory for $G$-actions \cite{DownarowiczZhang} (roughly speaking, if the infimum of the entropy functions in all symbolic extensions is affine, then it is attained). 
Once the above equivalences are proved, the final statement follows directly from \cite[Theorems 7.10 and 7.11]{DownarowiczZhang}.
\end{proof}
	
\section{Final comments and an example}\label{8}

In \cite{Zhou}, the author attempts to prove the tail variational principle directly (with the harder inequality valid only for essential partitions; as we already know, this would suffice for the general case as well). The problem occurs in the measure-theoretic considerations leading to the construction of an \im\ realizing the supremum appearing in the theorem. On page 1528 the author writes:
$$
\frac1{|F|}H_{\mu_F}(\P_l^F|\P_k^F)\le \frac1{|E|}H_{\nu_F}(\P_l^E|\P_k^E)+ \text{ small error term},
$$
where $F$ and $E$ are finite subsets of $G$ and $F$ is $(E,\delta)$-invariant for a small $\delta$. Unfortunately, the presence of the crucial coefficient $\frac1{|E|}$ does not follow in any way from the preceding calculations. Note, that the inequality resembles  Shearer's inequality\footnote{
A function $H$ on $\GF_G$ satisfies Shearer's inequality (adapted to our context) if whenever $F=\bigcup_{i=1}^n E_i$ and each element of $F$ is contained in at least $k$ of the sets $E_i$, then
$$
H(F)\le\frac1k\sum_{i=1}^nH(E_i).
$$
}.
In spite of the above mistake, this inequality would hold if the function $F\mapsto H_\mu(\P_l^F|\P_k^F)$ was {\bf strongly} subadditive on $\mathfrak F_G$. But, as we show in Example \ref{example} below, this is simply not true. Nonetheless, it is still possible that Shearer's inequality holds despite the lack of strong subadditivity (there are examples of this kind, see e.g.\ \cite[Example 6.4]{DFR}). In fact, the problem whether conditional Shannon entropy satisfies Shearer's inequality seems to be open (we failed to prove it or disprove it). But it is certain, that this inequality cannot be proved using plain subadditivity, as it is attempted in \cite{Zhou}, because Shearer's inequality is {\bf strictly stronger} than plain subadditivity (see e.g.\ \cite[Example 6.5]{DFR}). 

\begin{remark}
Example \ref{example} uses the finite group $\Z_3$. An example with an infinite group is easily obtained by adapting Example \ref{example} to $G=\Z_3\times\Z$ with an action that does not depend on the second coordinate of $G$. 
\end{remark}
\begin{exam}\label{example}
Conditional Shannon entropy is not strongly subadditive. The same example shows also that conditional counting entropy between covers is not strongly subadditive even when the considered covers are finite clopen partitions.
\end{exam}
\noindent And so, let $G= \Z_3$ and let $(X,G)$ be the full shift on two symbols, i.e.\ $X=\{a,b\}^{\Z_3}$ with the shift action. This action is generated by one map $\sigma$
defined by 
$$
\sigma((x_n)_{n\in\Z_3}) = (x_{(n+1\!\!\!\!\!\mod 3)})_{n\in\Z_3}.
$$ 
Let $\mu$ be the $(\frac12,\frac12)$-Bernoulli measure and let $\P$ denote the ``zero-coordinate'' partition $\P=\{[a],[b]\}$, where
$$
[a]=\{(x_n)_{n\in\Z_3}\in X:x_0=a\}
$$
and $[b]$ is defined analogously (or as the complement of $[a]$).
We also define 
$$
\Q=\P^{\{0,1\}}=\P\vee\sigma^{-1}(\P),\ \  \R=\P^{\{1,2\}}=\sigma^{-1}(\P)\vee\sigma^{-2}(\P).
$$
Consider the following three finite subsets of $\Z$:
$$
D=\{0\}, \ \ E=\{1\}, \ \ F=\{2\}. 
$$
Clearly,
$$
\Q^E= \P^{E\cup F}\ \text{and}\ \R^E= \P^{F\cup D},
$$ 
and
$$
\Q^{D\cup E\cup F}=\R^{D\cup E\cup F}=\Q^{D\cup E}=\R^{D\cup E}=\Q^{E\cup F}=\R^{E\cup F}= \P^{D\cup E\cup F}.
$$
Below we will just write $H(\cdot|\cdot)$ but the same applies to $H_\mu(\cdot|\cdot)$.

By a straightforward verification, we have:
\begin{gather*}
H(\Q^E|\R^E)=\log 2>0,\\
H(\Q^{D\cup E\cup F}|\R^{D\cup E\cup F})= H(\Q^{D\cup E}|\R^{D\cup E})= H(\Q^{E\cup F}| \R^{E\cup F})=0.
\end{gather*}
In particular, strong subadditivity fails, because of the strict inequality:
$$
H(\Q^{D\cup E\cup F}|\R^{D\cup E\cup F})\ \ >\ \ H(\Q^{D\cup E}|\R^{D\cup E})+ H(\Q^{E\cup F}|\R^{E\cup F}) - H(\Q^E|\R^E),
$$	
while $D\!\cup\!E\!\cup\!F$ and $E$ are respectively the union and intersection of the sets $D\!\cup\!E$ and $E\!\cup\!F$.

\section*{Acknowledgements}

The research of the first author is supported by the National Science Center, Poland (Grant
HARMONIA No.\ 2018/30/M/ST1/00061) and by the Wroc\l aw University of Science and
Technology.
The second author is supported by the National Natural Science Foundation of China (Grant No.\ 11731003).

\bibliographystyle{amsplain} % global bibliography

%\bibliography{DownZh}

%\iffalse

%\fi

\printindex

\end{document}